\documentclass[11pt]{amsart}
\usepackage[T1]{fontenc}
\usepackage{microtype}
\usepackage{amsmath,amssymb,amsthm,mathtools}
\usepackage{booktabs,tabularx,array,float}
\usepackage{caption}
\usepackage{xcolor}
\usepackage[colorlinks=true,linkcolor=blue!50!black,citecolor=blue!50!black,urlcolor=blue!50!black]{hyperref}
\newtheorem{theorem}{Theorem}
\newtheorem{proposition}[theorem]{Proposition}
\newtheorem{lemma}{Lemma}

\newcommand{\F}{\mathbb{F}}
\newcommand{\wt}{\operatorname{wt}}
\newcommand{\Span}{\operatorname{span}}

\title[A new lower bound for the kissing number in 19 dimensions]%
{A new lower bound for the\\ kissing number in 19 dimensions}
\author{Boon Suan Ho}
\date{}
\address{Department of Mathematics, National University of Singapore}
\email{\href{mailto:hbs@u.nus.edu}{hbs@u.nus.edu}}

\begin{document}

\begin{abstract}
We prove that the kissing number in 19 dimensions is at least 11948, improving the
bound of Cohn and Li by 256. Their odd-sign construction reduces the problem to finding
a binary code of length 19 and minimum distance 5 inside a 5-punctured extended binary
Golay code. In an explicit model $D$ of this ambient code, the 21 words of weights 3 and
4 are exactly the pairwise differences that such a code must avoid. They span a
10-dimensional subcode $K$. We exhibit a 6-dimensional subcode $M\le K$ for which all
21 words lie in only five cosets of $M$. After identifying $K/M$ with $\F_2^4$, these
five cosets are represented by
$\Sigma=\{e_1,e_2,e_3,e_4,e_1+e_2+e_3+e_4\}$. In the graph on $\F_2^4$ in which $u$
and $v$ are adjacent when $u+v\in\Sigma$, the set $\Sigma$ itself is independent.
Lifting its five vertices gives 320 words in $K$, and translating this set through the
four cosets of $K$ in $D$ gives a nonlinear code of size 1280 and minimum distance 5.
\end{abstract}

\maketitle
\thispagestyle{empty}

\section{Introduction}

The kissing number $k(n)$ is the largest number of unit spheres in $\mathbb R^n$ that
can touch a fixed unit sphere without overlapping. Equivalently, it is the largest size
of a subset of the unit sphere whose pairwise inner products are at most $1/2$. Many of
the strongest constructions in lower dimensions come from lattices and binary codes;
see, for example, \cite{Leech,LeechSloane,ConwaySloane}. Current tables of bounds are
available in \cite{CohnWeb,CohnTable}.

Cohn and Li improved the lower bounds in dimensions $17$ through $21$ by changing the
sign parity in classical code-based kissing configurations \cite{CohnLi}. In dimension
$19$, they begin with a 10668-point configuration and show that further points can be
indexed by binary words in the orthogonal complement of the code generated by the blocks
of the 5-shortened Steiner system $S(5,8,24)$. A collection of these points is compatible
whenever the corresponding words have pairwise Hamming distance at least $5$. The dual
of this shortened code is a 5-punctured extended binary Golay code, as explained in
Section~2. Cohn and Li use a
linear $[19,10,5]$ code of size $1024$, obtaining $k(19)\ge 11692$.

We replace their linear code by an explicit nonlinear code of size $1280$ in the same
ambient punctured Golay code. This proves the following theorem.

\begin{theorem}\label{thm:main}
The kissing number in 19 dimensions satisfies
\[
  k(19)\ge 11948.
\]
\end{theorem}

Here is an outline of our construction. Let $D$ be the explicit
5-punctured Golay code defined in Section~2.
\begin{enumerate}
\item A nonzero difference in $D$ is incompatible with minimum distance $5$ precisely
when it has weight $3$ or $4$. We denote the set of these 21 \emph{forbidden
differences} by $S$ and set $K=\Span(S)$.
\item We choose a 6-dimensional subcode $M\le K$ such that the 21 elements of $S$ occupy
only five cosets of $M$ in $K$.
\item Let $\pi:K\to K/M$ be the quotient map. The four classes
$s_1+M,\dots,s_4+M$ form a basis of the 4-dimensional vector space $K/M$. Under the
identification $s_i+M\leftrightarrow e_i$, the five cosets meeting $S$ are represented by
\[
  \pi(S)=\Sigma
  =\{e_1,e_2,e_3,e_4,e_1+e_2+e_3+e_4\}\subseteq\F_2^4.
\]
We therefore consider the graph on $\F_2^4$ in which distinct vertices $u$ and $v$ are
adjacent exactly when $u+v\in\Sigma$. An independent set in this graph selects quotient
classes whose full lifts to $K$ have no forbidden pairwise differences.
\item The set $\Sigma$ is itself a 5-element independent set. Its lift is a union $B$ of
five $M$-cosets and has size $5\cdot64=320$. Finally, $D/K$ has four elements, and
$S\subseteq K$, so the four translates of $B$ in the four cosets of $K$ are mutually
compatible. Their union is the required 1280-word code.
\end{enumerate}

This quotient also explains why the nonlinear construction beats the linear one. A linear selection of quotient classes
would be a subspace $L\le\F_2^4$ satisfying $L\cap\Sigma=\varnothing$. Suppose that
$\dim L\ge3$, and choose a 3-dimensional subspace $H\le L$. Then $H$ is a hyperplane,
so $H=\{x:a\mathbin{\cdot}x=0\}$ for some nonzero $a\in\F_2^4$. If a coordinate of $a$
is zero, the corresponding $e_i$ lies in $H$; if no coordinate is zero, then
$a=(1,1,1,1)$ and $e_1+e_2+e_3+e_4$ lies in $H$. Either conclusion contradicts
$L\cap\Sigma=\varnothing$. Hence $\dim L\le2$ and $|L|\le4$. In contrast, $\Sigma$ has
five elements and is independent because the sum of two distinct elements of $\Sigma$
never belongs to $\Sigma$. After lifting by $M$ and by $D/K$, the difference between four
and five quotient classes becomes the improvement from $1024$ to $1280$.

\section{The ambient punctured Golay code}

We identify a binary word with its support. Thus $\{6,7\}$ denotes the vector of
$\F_2^{19}$ with ones in coordinates $6$ and $7$. Addition of supports means symmetric
difference.

Consider the following twelve words of $\F_2^{19}$:
{\small
\[
\begin{array}{r@{\;=\;}l@{\qquad}r@{\;=\;}l}
 m_1 & \{1,8,9,12,16,17,18,19\}, &
 m_2 & \{2,10,11,14,15,17,18\}, \\
 m_3 & \{3,7,9,13,15,16,17,19\}, &
 m_4 & \{4,7,8,10,12,15,16,19\}, \\
 m_5 & \{5,10,12,13,15,16,17,18\}, &
 m_6 & \{6,7,8,9,10,13,16,18\}; \\
 s_1 & \{1,4,7,9\}, &
 s_2 & \{1,5,6,18\}, \\
 s_3 & \{1,3,12,15\}, &
 s_4 & \{1,10,13,19\}; \\
 r_1 & \{1,3,5,6,7,13,14,15,18\}, &
 r_2 & \{2,4,6,7,8,13,14,16,17,18\}.
\end{array}
\]
}
Define
\[
  D\coloneqq\Span\{m_1,\dots,m_6,s_1,s_2,s_3,s_4,r_1,r_2\}\le\F_2^{19}
\]
and
\[
  M\coloneqq\Span\{m_1,\dots,m_6\}.
\]
The role and origin of $M$ will be explained in Section~3. We first identify $D$ with a
punctured Golay code.

Let $G$ be the $12\times19$ matrix whose rows are the twelve displayed words in their
displayed order. To determine five columns that can be appended to $G$, write
$p_i\in\F_2^5$ for the appended entries in the row corresponding to $g_i$. For the
extended rows to be self-orthogonal, one must have
\[
  \langle p_i,p_j\rangle=\langle g_i,g_j\rangle
\]
for all $i,j$, while double evenness requires
$\wt(g_i)+\wt(p_i)\equiv0\pmod4$. Equivalently, for the matrix $P$ with rows $p_i$, one
seeks
\[
  PP^{\mathsf T}=GG^{\mathsf T}
\]
over $\F_2$, together with the row-weight congruences. A short backtracking search can
find such a matrix by choosing the 5-bit rows successively and imposing the row-weight
congruence and the required inner products with the rows already chosen. One solution is
displayed below. The proof checks the conditions directly and then identifies the
resulting code.

\begin{proposition}\label{prop:golay}
Let
\[
  P\coloneqq
  \begin{pmatrix}
  0&0&0&0&0\\
  0&0&0&0&1\\
  0&0&0&0&0\\
  0&0&0&0&0\\
  0&0&0&0&0\\
  0&0&0&0&0\\
  1&1&1&1&0\\
  0&1&1&1&1\\
  1&0&1&1&1\\
  1&1&0&1&1\\
  1&0&1&0&1\\
  0&0&1&1&0
  \end{pmatrix}.
\]
Write $\widetilde G=[G\mid P]$ and let $\widetilde D\le\F_2^{24}$ be its row span, with
the last five coordinates labelled $20,21,22,23,24$. Then $\widetilde D$ is a binary
doubly-even self-dual $[24,12,8]$ code. Consequently $\widetilde D$ is equivalent to the
extended binary Golay code, and $D$ is a 5-punctured extended binary Golay code.
\end{proposition}

\begin{proof}
Row reduction shows that the rows of $G$ are linearly independent. The rows of
$\widetilde G$ are therefore independent as well, so $\dim\widetilde D=12$. Direct
multiplication gives
\[
  \widetilde G\widetilde G^{\mathsf T}
  =GG^{\mathsf T}+PP^{\mathsf T}=0.
\]
Thus $\widetilde D$ is self-orthogonal, and its dimension implies that it is self-dual.
Each row of $\widetilde G$ has weight $8$ or $12$. If two orthogonal binary words have
weights divisible by $4$, then their sum also has weight divisible by $4$, because
\[
  \wt(u+v)=\wt(u)+\wt(v)-2|\operatorname{supp}(u)\cap\operatorname{supp}(v)|.
\]
It follows inductively that every word of $\widetilde D$ has weight divisible by $4$.

It remains to prove the minimum nonzero weight without enumerating all $4096$ words. The
first twelve columns of $\widetilde G$ are nonsingular. Row reduction therefore gives a
generator matrix $[I_{12}\mid A]$, where
{\scriptsize
\setlength{\arraycolsep}{2.5pt}
\[
A=\begin{pmatrix}
1&0&1&0&1&1&1&0&0&1&0&1\\
1&1&1&1&1&0&0&0&1&1&0&0\\
1&1&0&1&0&0&1&1&0&1&0&1\\
1&1&0&0&0&1&1&0&1&1&1&0\\
1&1&0&0&1&1&0&1&1&0&0&1\\
0&1&1&0&0&1&1&1&0&0&1&1\\
0&0&1&1&0&0&1&0&1&1&1&1\\
1&0&1&1&0&1&1&1&1&0&0&0\\
0&1&0&1&1&0&1&1&1&0&1&0\\
0&0&1&0&1&1&0&1&1&1&1&0\\
1&0&1&1&1&0&0&1&0&0&1&1\\
0&1&0&1&1&1&0&0&0&1&1&1
\end{pmatrix}.
\]
}
Self-orthogonality gives $I_{12}+AA^{\mathsf T}=0$, and hence
$AA^{\mathsf T}=I_{12}$. Inspection of the displayed matrix shows that every row and
every column of $A$ has weight $7$, and that the sum of two distinct rows has weight
$6$ or $10$.

Every codeword has the form $x[I_{12}\mid A]$ and therefore has weight
$\wt(x)+\wt(xA)$. Suppose a nonzero codeword had weight below $8$. Double evenness would
force its weight to be $4$, so $1\le\wt(x)\le4$. If $\wt(x)=1$, then
$\wt(x)+\wt(xA)=1+7=8$. If $\wt(x)=2$, the two-row calculation gives weight at least
$2+6=8$. If $\wt(x)=3$ and the total weight were $4$, then $xA=e_j$ for some $j$, and
\[
  x=xAA^{\mathsf T}=e_jA^{\mathsf T},
\]
which has weight $7$ because it is a column of $A$, a contradiction. If $\wt(x)=4$ and
the total weight were $4$, then $xA=0$, contradicting the invertibility of $A$. Thus the
minimum nonzero weight is at least $8$, and a row of $[I_{12}\mid A]$ has weight $8$, so
the minimum is exactly $8$.

Therefore $\widetilde D$ is a binary doubly-even self-dual $[24,12,8]$ code. By the
uniqueness of such a code \cite{ConwaySloane}, it is equivalent to the extended binary
Golay code. Deleting its last five coordinates recovers $D$.
\end{proof}

We record explicitly why the punctured Golay code is the ambient code in the Cohn--Li
construction. For a binary code $E\le\F_2^n$ and a set $J$ of coordinates, define
\[
\begin{aligned}
 \operatorname{pun}_J(E)
   &\coloneqq\{(x_i)_{i\notin J}:x\in E\},\\
 \operatorname{sh}_J(E)
   &\coloneqq\{(x_i)_{i\notin J}:x\in E,\ x_j=0\text{ for all }j\in J\}.
\end{aligned}
\]
These two operations satisfy
\begin{equation}\label{eq:short-punct}
  \operatorname{sh}_J(E)^\perp=\operatorname{pun}_J(E^\perp).
\end{equation}
The inclusion from right to left is immediate: a punctured word from $E^\perp$ is
orthogonal to every shortened word from $E$. For equality, let
$\rho_J:E\to\F_2^J$ be coordinate restriction and put
$r=\dim\rho_J(E)$. Then
\[
  \dim\operatorname{sh}_J(E)^\perp=n-|J|-\dim E+r.
\]
The kernel of puncturing on $E^\perp$ consists of the words supported on $J$ whose
restriction lies in $\rho_J(E)^\perp$, and therefore has dimension $|J|-r$. Hence
\[
  \dim\operatorname{pun}_J(E^\perp)
  =(n-\dim E)-(|J|-r)=n-|J|-\dim E+r,
\]
which proves \eqref{eq:short-punct}.

Let $G_{24}$ denote the extended binary Golay code and take $|J|=5$. The shortened code
is generated by the images of the octads disjoint from $J$, equivalently by the blocks of
the 5-shortened Steiner system used by Cohn and Li \cite[\S2]{CohnLi}. Applying
\eqref{eq:short-punct} and $G_{24}^\perp=G_{24}$ shows that its dual is
$\operatorname{pun}_J(G_{24})$, a 5-punctured Golay code.

Finally, the automorphism group $M_{24}$ is 5-transitive on the coordinate positions
\cite{ConwaySloane}. Hence any two choices of five punctured coordinates give equivalent
binary codes. The corresponding coordinate permutation of $\mathbb R^{19}$ is an
isometry and carries the entire odd-sign kissing construction to an isometric one.
Therefore it is enough to construct the required distance-5 code in the explicit code
$D$ above.

\section{The forbidden differences and the quotient graph}

In this paper, a \emph{low-weight word} means a word of weight $3$ or $4$. Since
$\widetilde D$ has minimum weight $8$, puncturing five coordinates shows that every
nonzero word of $D$ has weight at least $3$. Thus the set of differences that must be
avoided by a distance-5 subcode of $D$ is exactly
\[
  S\coloneqq\{x\in D:\wt(x)\in\{3,4\}\}.
\]
Let $\Gamma$ be the graph with vertex set $D$ in which distinct $x,y$ are adjacent when
$x+y\in S$. A subset of $D$ is independent in $\Gamma$ exactly when it is a binary code
of minimum distance at least $5$. Set
\[
  K\coloneqq\Span(S)
\]
and define
\[
  s_5\coloneqq\{3,5,7,10\}.
\]
The identity
\[
  s_5=s_1+s_2+s_3+s_4+m_4+m_6
\]
will identify its class modulo $M$.

The choice of $M$ was first found by computer search. After enumeration gave
$\dim K=10$, the search considered 6-dimensional subspaces $M\le K$ disjoint from $S$
and found one for which $S$ met five cosets.

\begin{table}[t]
\centering
\small
\renewcommand{\arraystretch}{1.05}
\caption{The 21 low-weight words in $D$. The twenty words of weight $4$ are grouped
into five sets of four words with pairwise disjoint supports; the unique word of weight
$3$ is placed in the first row. Proposition~\ref{prop:structure} shows that these five
rows correspond exactly to the five $M$-cosets that meet $S$.}
\label{tab:lowweight}
\begin{tabularx}{\textwidth}{@{}l>{\raggedright\arraybackslash}X@{}}
\toprule
Coset in $K/M$ & Low-weight words in that coset \\
\midrule
$s_1+M$ & $\{2,11,14\}$, $\{1,4,7,9\}$, $\{3,6,8,19\}$, $\{5,12,13,16\}$, $\{10,15,17,18\}$ \\
$s_2+M$ & $\{1,5,6,18\}$, $\{3,9,13,17\}$, $\{4,8,10,12\}$, $\{7,15,16,19\}$ \\
$s_3+M$ & $\{1,3,12,15\}$, $\{4,5,17,19\}$, $\{6,9,10,16\}$, $\{7,8,13,18\}$ \\
$s_4+M$ & $\{1,10,13,19\}$, $\{3,4,16,18\}$, $\{5,8,9,15\}$, $\{6,7,12,17\}$ \\
$s_5+M$ & $\{1,8,16,17\}$, $\{3,5,7,10\}$, $\{4,6,13,15\}$, $\{9,12,18,19\}$ \\
\bottomrule
\end{tabularx}
\end{table}

This choice leaves a 4-dimensional quotient with only 16 classes to inspect, while each
class lifts to $64$ words. Five compatible classes therefore give $320$ words in $K$,
and their translates give the same number in each coset of $K$; four classes would give
only $256$.

The subspace returned by the search can be described directly from the low-weight words.
Let $t=\{2,11,14\}$ be the unique word of weight $3$. On the twenty weight-4 words, join
two words when their supports are disjoint. The resulting graph has five connected
components $Q_1,\dots,Q_5$, each a complete graph on four vertices; these are the five
sets displayed in the rows of Table~\ref{tab:lowweight}. Let
\[
  U\coloneqq\Span\{q+q':q,q'\in Q_i\text{ for some }i\}.
\]
A row reduction gives $\dim U=5$. If $q\in Q_1$, then another row reduction shows that
\[
  U+\Span\{t+q\}=\Span\{m_1,\dots,m_6\}=M.
\]
Thus all four words in each $Q_i$ are congruent modulo $M$, and $t$ is congruent to the
words in $Q_1$. Hence the 21 low-weight words occupy exactly five classes modulo $M$.

\begin{proposition}\label{prop:structure}
Let $\pi:K\to K/M$ be the quotient map. Then:
\begin{enumerate}
\item $S$ consists exactly of the 21 words in Table~\ref{tab:lowweight}.
\item $\dim M=6$ and
\[
  K=\Span(M,s_1,s_2,s_3,s_4),\qquad \dim K=10.
\]
\item The four classes $s_1+M,s_2+M,s_3+M,s_4+M$ form a basis of $K/M$.
\item Under the identification $s_i+M\leftrightarrow e_i$, the image of the forbidden
set is
\[
  \pi(S)=\Sigma
  \coloneqq\{e_1,e_2,e_3,e_4,e_1+e_2+e_3+e_4\}.
\]
Equivalently, the five $M$-cosets meeting $S$ are
$s_1+M,\dots,s_5+M$.
\item The images of $r_1,r_2$ form a basis of $D/K$, and
\[
  D=K\mathbin{\dot\cup}(K+r_1)\mathbin{\dot\cup}(K+r_2)
    \mathbin{\dot\cup}(K+r_1+r_2).
\]
\end{enumerate}
\end{proposition}

\begin{proof}
Enumerating the $4096$ words of $D$ gives exactly the 21 low-weight words in
Table~\ref{tab:lowweight}. Direct row reduction against $m_1,\dots,m_6$ verifies the coset
labels in the table; in particular, $M\cap S=\varnothing$. This proves (1).

The following identities express the six generators of $M$ as sums of words from the
table:
\begin{align*}
  m_1&=\{1,8,16,17\}+\{9,12,18,19\}, &
  m_2&=\{2,11,14\}+\{10,15,17,18\}, \\
  m_3&=\{3,9,13,17\}+\{7,15,16,19\}, &
  m_4&=\{4,8,10,12\}+\{7,15,16,19\}, \\
  m_5&=\{5,12,13,16\}+\{10,15,17,18\}, &
  m_6&=\{6,9,10,16\}+\{7,8,13,18\}.
\end{align*}
Hence $M\le K$, and $s_1,\dots,s_4\in S$ imply
$\Span(M,s_1,s_2,s_3,s_4)\le K$. Conversely, every word of $S$ belongs to one of the
five cosets displayed in the table, and
$s_5=s_1+s_2+s_3+s_4+m_4+m_6$. Thus every element of $S$ lies in
$\Span(M,s_1,s_2,s_3,s_4)$, proving equality with $K$.

A row reduction shows that $m_1,\dots,m_6$ are independent and that
$s_1+M,\dots,s_4+M$ are independent in $K/M$. This proves (2) and (3). The table and the
formula for $s_5$ then give
\[
  \pi(S)=\{e_1,e_2,e_3,e_4,e_1+e_2+e_3+e_4\},
\]
which proves (4). The same row reduction shows that $r_1+K$ and $r_2+K$ are independent
in $D/K$. Since $\dim D=12$ and $\dim K=10$, they form a basis, proving (5).
\end{proof}

To decide which $M$-cosets can be lifted together, define a graph $\mathcal Q$ on
$K/M$ by joining distinct classes $u$ and $v$ when $u+v\in\pi(S)$. Under the basis in
Proposition~\ref{prop:structure}, this is the graph on
$\F_2^4$ with adjacency rule
\[
  u\sim v\quad\Longleftrightarrow\quad u+v\in\Sigma.
\]
This 16-vertex graph is a standard model of the Clebsch graph, but only the displayed
adjacency rule will be used.

\begin{lemma}\label{lem:quotient-lift}
If $J\subseteq K/M$ is an independent set in $\mathcal Q$, then its full inverse image
$\pi^{-1}(J)$ is independent in the induced graph $\Gamma|_K$.
\end{lemma}

\begin{proof}
Suppose $x,y\in\pi^{-1}(J)$ and $x+y\in S$. Then
\[
  \pi(x)+\pi(y)=\pi(x+y)\in\pi(S)=\Sigma.
\]
Thus $\pi(x)$ and $\pi(y)$ would be adjacent in $\mathcal Q$. They cannot be equal,
because then their sum would be $0\notin\Sigma$. This contradicts the independence of
$J$.
\end{proof}

\section{Lifting an independent set to a 1280-word code}

\begin{lemma}\label{lem:Iindependent}
The set
\[
  I\coloneqq\Sigma
  =\{e_1,e_2,e_3,e_4,e_1+e_2+e_3+e_4\}
\]
is a 5-element independent set in $\mathcal Q$.
\end{lemma}

\begin{proof}
The sum of two distinct basis vectors is not in $\Sigma$. The sum of a basis vector with
$e_1+e_2+e_3+e_4$ is the sum of the other three basis vectors, which is also not in
$\Sigma$. Therefore no two distinct elements of $I$ are adjacent.
\end{proof}

\begin{proposition}\label{prop:Acode}
Let
\[
  B\coloneqq(s_1+M)\cup(s_2+M)\cup(s_3+M)\cup(s_4+M)\cup(s_5+M)
\]
and
\[
  A\coloneqq B\cup(B+r_1)\cup(B+r_2)\cup(B+r_1+r_2).
\]
Then $A\subseteq D$ has size $1280$ and minimum distance $5$.
\end{proposition}

\begin{proof}
The five quotient classes in $I$ are distinct, so
\[
  |B|=5|M|=5\cdot64=320.
\]
Moreover, $B=\pi^{-1}(I)$ inside $K$, and Lemmas~\ref{lem:quotient-lift}
and~\ref{lem:Iindependent} show that $B$ has no forbidden pairwise differences. Hence
$B$ has minimum distance at least $5$.

By Proposition~\ref{prop:structure}(5), the four translates in the definition of $A$ lie
in distinct cosets of $K$. Since every forbidden difference belongs to $S\subseteq K$,
two words in different cosets of $K$ cannot have a forbidden difference. Thus $A$ has
minimum distance at least $5$, and
\[
  |A|=4|B|=4\cdot320=1280.
\]
The minimum distance is exactly $5$: the words $s_2$ and $s_3+r_1$ belong to $A$ and
differ by $\{1,7,12,13,14\}$.
\end{proof}

\section{Consequence for kissing numbers}

For $c\in D$, set
\[
  v(c)\coloneqq\sqrt{\frac{8}{19}}\,
  \bigl((-1)^{c_1},\dots,(-1)^{c_{19}}\bigr)\in\mathbb R^{19}.
\]
Let $C\le\F_2^{19}$ be the code generated by the blocks of the 5-shortened Steiner
system. Cohn and Li show that the odd-sign construction in dimension $19$ gives a
10668-point kissing configuration and that one may adjoin all $v(c)$ with $c$ in any
minimum-distance-5 code contained in $C^\perp$ \cite[\S2]{CohnLi}. By
\eqref{eq:short-punct}, $C^\perp$ is a 5-punctured extended binary Golay code. The
coordinate-equivalence discussion after Proposition~\ref{prop:golay} permits us to use
our explicit model $D$, and Proposition~\ref{prop:Acode} supplies a 1280-word code in
$D$. Therefore
\[
  k(19)\ge10668+1280=11948,
\]
which proves Theorem~\ref{thm:main}.


\section{Why the dimension-19 gain is exceptional}

For $c,c'\in\F_2^n$ at Hamming distance $d$, the corresponding sign vectors satisfy
\[
 \left\langle \sqrt{\frac8n}\bigl((-1)^{c_i}\bigr)_{i=1}^n,
 \sqrt{\frac8n}\bigl((-1)^{c'_i}\bigr)_{i=1}^n\right\rangle
 =\frac8n(n-2d).
\]
This inner product is at most $4$ exactly when $d\ge n/4$ \cite[\S2]{CohnLi}. Thus in
dimensions $18$, $19$, and $20$ the relevant integer condition is minimum distance at
least $5$. In a punctured Golay code, one may therefore form a graph by joining words
whose difference has weight below $5$; compatible additional sign vectors are independent
sets in this graph.

\subsection{Dimension 20}
Let $T$ be the 4-element set of punctured coordinates, and let $D_{20}$ be the resulting
4-punctured extended Golay code. Puncturing is injective, so $|D_{20}|=4096$, and every
nonzero word has weight at least $4$. The forbidden words are precisely the weight-4 words
obtained from octads containing $T$.

There are exactly five such octads. Indeed, each of the twenty coordinates outside $T$
completes $T$ to a 5-element set, and that 5-element set lies in a unique octad. Each octad
containing $T$ is counted four times, once for each of its coordinates outside $T$, so the
number of octads is $20/4=5$. Denote them by $O_1,\dots,O_5$, and let $h_i$ be the
weight-4 word supported on $O_i\setminus T$. If two of these supports had a common
coordinate, the corresponding octads would contain the same 5-element set, which is
impossible. Hence the supports of $h_1,\dots,h_5$ are pairwise disjoint, and these five
words are linearly independent.

Let $H=\operatorname{span}\{h_1,\dots,h_5\}$. The connected component of $0$ in the forbidden-difference graph is $H$, and every other
component is a coset of $H$. Within each component, adding one of the five basis words
$h_i$ changes exactly one basis coordinate. Thus every component is a 5-dimensional cube,
and there are $2^{12-5}$ components. The 5-dimensional cube has independence number
$16$: one parity class gives 16 independent vertices, while a perfect matching in any
coordinate direction shows that no independent set can be larger. Hence the whole graph
has independence number
\[
  2^{12-5}\cdot16=2048.
\]
This is exactly the size of the linear code used by Cohn and Li. The cube decomposition
therefore leaves no nonlinear gain in this direct 4-punctured model.

\subsection{Dimension 19}
After one further puncture, the forbidden set is no longer a linearly independent basis.
It has 21 elements spanning a 10-dimensional space, and the special 6-dimensional
subspace $M$ places these elements in only five quotient classes. The resulting graph
on $K/M$ has a 5-element independent set, although every independent linear subspace has
at most four elements. The improvement comes from the dependencies among the forbidden
words: they permit this quotient reduction, and the quotient graph admits a larger
nonlinear selection than any linear subspace.

\subsection{Dimension 18}
The direct 6-punctured analogue is different from the actual 18-dimensional construction
of Cohn and Li, which uses an extension of the odd 16-dimensional configuration and
additional compatibility conditions \cite[\S4]{CohnLi}. It is nevertheless useful for
showing why the punctured-Golay method does not reproduce their bound.

Let $J$ be the 6-element set of punctured coordinates, and let $D_{18,J}$ be the code
obtained by puncturing the extended binary Golay code in those coordinates. The
automorphism group of the Golay code has two orbits on such sets $J$ \cite{ConwaySloane}.
We call them \textbf{Case 1}, when $J$ is not contained in an octad, and \textbf{Case 2},
when $J$ is contained in an octad. At most one octad can contain $J$, because two distinct
octads cannot contain the same 5-element subset. In the coordinate model used above,
representatives of the two cases are $\{1,2,3,4,5,6\}$ and
$\{1,2,3,4,5,8\}$, respectively.

For $0\le i\le5$, the number of octads containing a fixed $i$-element subset is
$\lambda_i=\binom{24-i}{5-i}/\binom{8-i}{5-i}$, and
\[
  (\lambda_0,\dots,\lambda_5)=(759,253,77,21,5,1).
\]
Write $\varepsilon_J=0$ in Case 1 and $\varepsilon_J=1$ in Case 2. Inclusion--exclusion
then gives
\[
  \#\{\text{octads disjoint from }J\}
  =\sum_{i=0}^{5}(-1)^i\binom6i\lambda_i+\varepsilon_J
  =45+\varepsilon_J.
\]
Thus 45 octads are disjoint from $J$ in Case 1, and 46 are disjoint from $J$ in Case 2.

Puncturing six coordinates is injective because the extended Golay code has minimum
distance $8$. Consequently $|D_{18,J}|=4096$, and only punctured octads can have weight
below $5$. Let $N_r$ be the number of octads meeting $J$ in exactly $r$ coordinates.
Only $N_4,N_5,N_6$ are relevant, since such octads puncture to words of weights $4,3,2$,
respectively.

First, $N_6=\varepsilon_J$. Each of the six 5-element subsets of $J$ lies in a unique
octad. In Case 1 these six octads are distinct and meet $J$ in exactly five coordinates,
so $N_5=6$. In Case 2 all six 5-element subsets lie in the unique octad containing $J$,
so $N_5=0$. To find $N_4$, count pairs $(R,O)$ for which $R$ is a 4-element subset of
$J$ and $O$ is an octad containing $R$. There are
$\binom64\lambda_4=15\cdot5=75$ such pairs. An octad meeting $J$ in four, five, or six
coordinates contributes $1$, $5$, or $15$ pairs, respectively. Hence
\[
  75=N_4+5N_5+15N_6.
\]
It follows that the numbers of forbidden words of weights $2$, $3$, and $4$ are
$(0,6,45)$ in Case 1 and $(1,0,60)$ in Case 2. Let
$S_J=\{x\in D_{18,J}:1\le\wt(x)\le4\}$, and let $\Gamma_J$ be the graph on
$D_{18,J}$ in which $x$ and $y$ are adjacent when $x+y\in S_J$. The graph has degree
$51$ in Case 1 and degree $61$ in Case 2.

For completeness, we describe the spectral calculation used to bound its independence
number. Let $G_J$ be the $12\times18$ generator matrix obtained by deleting the columns
indexed by $J$ from $\widetilde G$. Since puncturing is injective, every word of
$D_{18,J}$ has a unique expression $xG_J$ with $x\in\F_2^{12}$. For
$a\in\F_2^{12}$, define
\[
  \chi_a(xG_J)=(-1)^{a\mathbin{\cdot}x}.
\]
These are the 4096 characters of $D_{18,J}$, and the adjacency eigenvalue associated with
$\chi_a$ is
\[
  \theta_a=\sum_{s\in S_J}\chi_a(s).
\]
Exact evaluation of these integer sums gives least eigenvalue $-19$ with multiplicity
$6$ in Case 1, and least eigenvalue $-21$ with multiplicity $12$ in Case 2.
The Hoffman ratio bound \cite{Hoffman} for a $d$-regular graph on $v$ vertices with least
eigenvalue $\tau<0$ is
\[
  \alpha\le\frac{v(-\tau)}{d-\tau}.
\]
Therefore
\[
\begin{aligned}
  \text{Case 1:}\qquad
  \alpha(\Gamma_J)
  &\le \left\lfloor\frac{4096\cdot19}{51+19}\right\rfloor=1111,\\
  \text{Case 2:}\qquad
  \alpha(\Gamma_J)
  &\le \left\lfloor\frac{4096\cdot21}{61+21}\right\rfloor=1048.
\end{aligned}
\]

The corresponding odd-sign base configuration consists of
$4\binom{18}{2}$ root vectors and 128 vectors for each octad disjoint from $J$. Its size
is therefore
\[
  4\binom{18}{2}+128\cdot45=6372
  \quad\text{in Case 1},
  \qquad
  4\binom{18}{2}+128\cdot46=6500
  \quad\text{in Case 2}.
\]
After adjoining the sign vectors indexed by an independent set, the two cases give at
most $6372+1111=7483$ and $6500+1048=7548$ points, respectively. Both totals are below
the Cohn--Li bound $7654$.

The three neighboring cases differ as follows. In dimension $20$, the forbidden words
form a basis and the graph splits into cubes, whose optimal independent sets are already
linear. In dimension $18$, the forbidden graphs are dense enough that their spectra rule
out a competitive independent set. Dimension $19$ lies between these cases: its forbidden
words have enough dependencies to admit a small quotient, and that quotient has a
nonlinear independent set strictly larger than any linear one.

\section*{Acknowledgements}
The author thanks the anonymous referee for a careful reading and detailed comments that
improved the exposition. The author used ChatGPT (OpenAI) while exploring possible proof
strategies and revising the exposition. The author independently verified all mathematical
claims and takes full responsibility for the manuscript.

\section*{Statements and Declarations}

\subsection*{Data availability}
The explicit 11948-point configuration and the verification scripts are available at
\url{https://github.com/boonsuan/kissing}.

\subsection*{Competing interests}
The author declares that there are no competing interests.

\end{document}